\documentclass[10pt]{article}

\usepackage{amsmath,amssymb,amsbsy,amsthm,color,epsfig,mathrsfs}
\usepackage{amsbsy,floatrow}
\usepackage{caption}
\usepackage{pifont}

\usepackage{hyperref}
\usepackage{bookmark}

\renewcommand{\Im}{\,\mathrm{Im}}

\newtheorem{theorem}{Theorem}

\newtheorem{lemma}[theorem]{Lemma}

\newtheorem{proposition}[theorem]{Proposition}
\newtheorem{remark}[theorem]{Remark}

\newcounter{spslist}
\newenvironment{spslist}{
  \begin{list}
  {\arabic{spslist}.}
  {\usecounter{spslist}
  \setlength{\leftmargin}{2.8em}
  \setlength{\labelsep}{0.6em}
  \setlength{\labelwidth}{1em}
  \setlength{\topsep}{1ex}
  \setlength{\rightmargin}{0em}
  \setlength{\itemsep}{0.5ex}
  \setlength{\parsep}{0em}
  \setlength{\itemindent}{0em} }}
  {\end{list}}

\newcommand{\mat}[5]{ \renewcommand{\arraystretch}{#1}
                    \left[\! \begin{array}{cc}
                            #2 & #3 \\
                            #4 & #5 \end{array} \!\right] }

\newcounter{geqncount}
    {\refstepcounter{equation}%
     \setcounter{geqncount}{\value{equation}}%
     \setcounter{equation}{0}%
  }%
    {\setcounter{equation}{\value{geqncount}}}
\newcommand{\tensor}{\!\otimes\!}

\newcommand{\CC}{\mathbb{C}}
\newcommand{\NN}{\mathbb{N}}

\newcommand{\ZZ}{\mathbb{Z}}
\newcommand{\RR}{\mathbb{R}}

\newcommand{\E}{\mathcal{E}}
\newcommand{\Ehalf}{\mathcal{E}^{\hspace{-1pt}\half}\!}
\newcommand{\EhalfC}{\mathcal{E}^\half_\CC\!}
\newcommand{\onehalf}{{\textstyle{\frac{1}{2}}}}
\newcommand{\half}{{  \raisebox{1pt}{  \hspace{-3pt}\tiny $\textstyle{\frac{1}{2}}$\hspace{-1pt}  }   }}
\newcommand{\mhalf}{{  \raisebox{1pt}{  \hspace{-3pt}\tiny --$\textstyle{\frac{1}{2}}$\hspace{-1pt}  }   }}

\newcommand{\st}{\hspace{4pt}{\text{\large$\mid$}}\hspace{4pt}}
\newcommand{\one}{\mathrm{1}}
\newcommand{\Span}{\mathrm{span}}

\def\cases#1{\left\{\begin{matrix} #1 \end{matrix}\right.}

\def\nonpmatrix#1{\begin{matrix} #1 \end{matrix}}
\def\Imag{\mathop{\rm Im}}

\begin{document}

\bibliographystyle{plain}

\begin{center}
{\bf \Large  The inverse problem for a spectral asymmetry function\\\vspace{3pt} of the Schr\"odinger operator on a finite interval}
\end{center}

\vspace{0ex}

\begin{center}
{\scshape \large B. Malcolm Brown\footnote{School of Computer Science and Informatics, Cardiff University, BrownBM@cardiff.ac.uk} \,and\, Karl Michael Schmidt\footnote{School of Mathematics, Cardiff University, SchmidtKM@cardiff.ac.uk}\\ \,and\, Stephen P\hspace{-2.7pt}. Shipman\footnote{Department of Mathematics, Louisiana State University @ Baton Rouge, shipman@lsu.edu, {\sc orcid} 0000-0001-6620-6528} \,and\, Ian Wood\footnote{School of Mathematics, Statistics and Actuarial Science, University of Kent, i.wood@kent.ac.uk}} \\
\vspace{3ex}
{\itshape 
${}^{1,2}$\hspace{-1pt}Cardiff University \,and\, ${}^3$\hspace{-1pt}Louisiana State University \,and\, ${}^4$\hspace{-1pt}University of Kent}
\end{center}

\vspace{3ex}
\centerline{\parbox{0.9\textwidth}{
{\bf Abstract.}\
For the Schr\"odinger equation $-d^2 u/dx^2 + q(x)u = \lambda u$ on a finite $x$-interval, there is defined an ``asymmetry function" $a(\lambda;q)$, which is entire of order $1/2$ and type $1$ in~$\lambda$.  Our main result identifies the classes of square-integrable potentials $q(x)$ that possess a common asymmetry function.  For any given $a(\lambda)$, there is one potential for each Dirichlet spectral sequence.
}}

\vspace{3ex}
\noindent
\begin{mbox}
{\bf Key words:}  spectral theory; Schr\"odinger operator; inverse spectral problem; entire function; asymmetry function
\end{mbox}

\vspace{3ex}

\noindent
\begin{mbox}
{\bf MSC:}  34A55, 34L05, 30E05
\end{mbox}
\vspace{3ex}

%34A55  	Inverse problems involving ordinary differential equations
%34L05  	General spectral theory of ordinary differential operators
%34L15  	Eigenvalues, estimation of eigenvalues, upper and lower bounds of ordinary differential operators
%34B40  	Boundary value problems on infinite intervals for ordinary differential equations
%30E05  	Moment problems and interpolation problems in the complex plane

\hrule

\section{Introduction} %%%%%%%%%%%%%%%%%%%%%

Consider the spectral Schr\"odinger equation $-d^2 u/dx^2+q(x)u = \lambda u$\, on the $x$-interval $[0,1]$ with $q$ real valued and square-integrable.  Let $c(x,\lambda;q)$ and $s(x,\lambda;q)$ be a pair of fundamental solutions satisfying
\begin{equation*}
\begin{split}
  c(0,\lambda;q) = 1, &\quad s(0,\lambda;q) = 0, \\
  c'(0,\lambda;q) = 0, &\quad s'(0,\lambda;q) = 1,
\end{split}
\end{equation*}
in which the prime denotes differentiation with respect to~$x$.  In this article, we are interested in the following entire spectral function associated with~$q$:
\begin{equation*}
  a(\lambda;q) \;:=\; \textstyle\frac{1}{2} \big( c(1,\lambda;q) - s'(1,\lambda;q) \big).
\end{equation*}
This is called the spectral asymmetry function for the potential $q(x)$, or simply its asymmetry function.
When the potential is understood, we may suppress writing the dependence of these functions on~$q$.

The connection of the asymmetry function to asymmetry of the potential is seen as follows.  Define $\tilde q(x) = q(1-x)$; then, by considering the transfer function taking Cauchy data at $0$ to Cauchy data at $1$, we see that $c(1,\lambda;\tilde q)=s'(1,\lambda;q)$, so that
\begin{equation*}
  a(\lambda;q) \;:=\; \textstyle\frac{1}{2} \big( c(1,\lambda;q) - c(1,\lambda;\tilde q)  \big).
\end{equation*}
This shows that the asymmetry function vanishes identically as a function of $\lambda$ whenever $q$ is symmetric about the midpoint of $[0,1]$.  In fact, $a(\lambda;q)$ vanishes identically if and only if $q$ is symmetric.  An abbreviated proof of this appears in~\cite[Lemma~4]{Yurko1975}; a more detailed proof is given in \cite[Theorem\,2]{Shipman2019}, along with several other properties of the asymmetry function.  This motivates the idea of the  {\em asymmetry class of potentials} associated to a given asymmetry function $a(\lambda)$---it is the set of all potentials $q$ that possess $a(\lambda)$ as its asymmetry function, that is, all $q$ such that $a(\lambda;q)=a(\lambda)$. 

Asymmetry classes play a key role in the spectral theory of bilayer graph operators in~\cite{Shipman2019}.  
Their fundamental property is that the Dirichlet-to-Neumann maps for two different potentials commute if and only if the potentials have the same asymmetry function.  The DtN map is given by the matrix
\begin{equation*}
  N(\lambda;q) \;=\; \frac{1}{s(1,\lambda;q)} \mat{1.2}{-c(1,\lambda;q)}{1}{1}{-s'(1,\lambda;q)},
\end{equation*}
which maps Dirichlet data of a solution of $-d^2 u/dx^2+q(x)u = \lambda u$ on $[0,1]$ to the Neumann data of the solution: $N(\lambda)[u(0),u(1)]^t = [u'(0),-u'(1)]^t$.  It is a meromorphic function with poles at the roots of $s(\lambda;q)$, which are the Dirichlet eigenvalues of the Schr\"odinger operator $-d^2/dx^2+q(x)$ on $[0,1]$. 

The asymmetry function and the function $b(\lambda)=\textstyle\frac{1}{2} \big( c(1,\lambda;q) + c(1,\lambda;\tilde q) \big)$ coincide with the functions $u_-(\lambda)$ and $u_+(\lambda)$ in~\cite[p.\,494;\,Lemma\,4.1]{MarcenkoOstrovskii1975}, where the authors characterize the spectra of Hill's operator ($-d^2/dx^2+q(x)$ with periodic potential) as certain admissible sequences of intervals on the line.
The functions $a$ and $b$ are also identical to $\delta$ and $\Delta$, respectively, in~\cite[p.\,2]{Yurko2016}.

The purpose of the present work is to characterize the set of all asymmetry functions and, given a fixed asymmetry function $a$, to identify its asymmetry class of potentials, which, as defined above, is the set of all potentials $q$ that possess that $a$ as its asymmetry function.  The main result is Theorem~\ref{thm:main} in section~\ref{sec:bijection}, which states that, for each asymmetry function $a$ in a certain Hilbert space of entire functions, the asymmetry class of $a$ contains one potential for each admissible sequence of Dirichlet eigenvalues.

The analysis in this paper draws primarily upon deep work in spectral theory of the Schr\"odinger operator on the interval by E. Trubowitz, H. P. McKean, and J. P\"oschel.

\smallskip In what follows, to avoid cumbersome notation, we will often suppress the dependence of functions on the potential $q$. 
The following function spaces will be used in this article.
\begin{align*}
  L^2[0,1] &=  \Big\{ q:[0,1]\to\RR \st \int_0^1\! q(x)^2\,dx <\infty \Big\} \\
  \ell^2(\ZZ) &= \Big\{ \alpha:\ZZ\to\RR \st \sum_{n=-\infty}^\infty \alpha_n^2 < \infty \Big\}\\
  \mathcal{S} &= \Big\{ \alpha\in\ell^2 \st (\pi^2n^2 + \alpha_n)_{n\in\NN} \text{ is strictly increasing} \Big\} \\
  \ell^2_1(\NN) &= \Big\{ \alpha:\NN\to\RR \st \sum_{n=1}^\infty n^2\alpha_n^2 < \infty \Big\}\\
  \E &= \Big\{ \phi : \CC\to\CC \st \phi\;\text{ entire, order $\leq 1$, type $\leq 1$, } \phi(\RR)\subset\RR,\, \int_\RR |\phi(\lambda)|^2d\lambda <\infty \Big\}\\
  \Ehalf &= \Big\{ \phi : \CC\to\CC \st \phi\;\text{ entire, order $\leq 1/2$, type $\leq 1$, } \phi(\RR)\subset\RR,\, \int_0^\infty |\phi(\lambda)|^2\lambda^{\hspace{-1pt}\half} d\lambda <\infty \Big\}
\end{align*}

\section{Basic properties of the asymmetry function} %%%%%%%%%%%%%%%%%%%%%

The property of the asymmetry function most relevant to this paper is the fact that it is in the class $\Ehalf$ and that its evaluation on a Dirichlet spectral sequence is in $\ell^2_1$.

\begin{proposition}\label{prop:a}
Let $q\in L^2$ and a sequence $(\mu_n)_{n\in\NN}$, $\mu_n = n^2\pi^2 + c_n$, with $(c_n)_{n\in\NN}$ a real-valued bounded sequence, be given.  The asymmetry function $a(\lambda;q)$ lies in the class~$\Ehalf$, and the sequence $(a(\mu_n;q))_{n\in\NN}$ lies in the class $\ell^2_1$.
\end{proposition}

\begin{proof}
Since $q$ is real-valued, both $c(x,\lambda;q)$ and $s(x,\lambda;q)$ are real-valued for real $\lambda$, and thus $a$ is real-valued on the real line.
The function $c(1,\lambda;q)$ is an entire function of $\lambda$ of order $1/2$ satisfying
\begin{equation}\label{cpert}
  c(1,\lambda) \;=\; \cos\sqrt{\lambda} + \lambda^\mhalf c_1(\lambda) + R(\lambda),
\end{equation}
where
\begin{equation}
  c_1(\lambda) = \onehalf\sin\sqrt{\lambda} \int_0^1 q(x)\,dx
      + \onehalf \int_0^1 \sin\big(\sqrt{\lambda}(1-2x)\big) q(x) \,dx
\end{equation}
and $R$ is an entire function such that $z\,e^{-|\Imag\sqrt z|}\,R(z)$ $(z\in\mathbb{C})$ is bounded (cf.\ \cite[(1.1.15)]{FreilingYurko2001}; \cite[pp.\,14--15]{PoschelTrubowitz1987}).
The function $c(\tilde{q};1,\lambda)$ can be written identically to (\ref{cpert}) except with $q(x)$ replaced with $q(1-x)$ in the expression of $c_1(\lambda)$.  Thus we obtain
\begin{equation}\label{afourier}
\begin{split}
  a(\lambda) &= \frac{1}{2\sqrt{\lambda}} \int_0^1 \sin\big( \sqrt{\lambda}(1-2x) \big) q_\mathrm{o}(x) \,dx
      + \tilde R(\lambda)\\
      &= \frac{-1}{4\sqrt{\lambda}} \int_{-1}^1 \sin(\sqrt{\lambda}\,y)\, q_\mathrm{o}\big(\textstyle\frac{y+1}{2}\big) \,dy
      + \tilde R(\lambda),
\end{split}
\end{equation}
in which $q_\mathrm{o}(x) = (q(x) - q(1-x))/2$ is the odd part of $q$, and $\tilde R$ is an entire function such that $z\,e^{-|\Imag\sqrt z|}\,\tilde R(z)$ $(z\in\mathbb{C})$ is bounded.
If we abbreviate
$$ f(z) = - \frac 1 4 \int_{-1}^1 \sin(z y)\, q_\mathrm{o}\textstyle\big(\frac{y+1}2\big)\,d y \qquad (z \in \mathbb{C}), $$
then by the Paley-Wiener theorem, $f$ is an entire function of exponential type 1 that is
square-integrable over parallels to the real axis. Hence, setting $\tilde a(\lambda) = f(\sqrt\lambda)/\sqrt\lambda$, we have
$$ \int_0^\infty |\tilde a(\lambda)|^2 \sqrt\lambda\,d\lambda = \int_{-\infty}^\infty |\tilde a(z^2)\,z|^2\,d z < \infty. $$
Thus $\tilde a \in \Ehalf$. 
Moreover, $\tilde R \in \Ehalf$, and as $\Ehalf$ is a vector space, it follows that $a = \tilde a + \tilde R \in \Ehalf$.

For the given sequence $\mu_n$, $\sqrt{\mu_n} = n\pi + O(n^{-1})$ and $\sqrt{\mu_n}^{\ -1} = (n\pi)^{-1} + O(n^{-3})$; and evaluation of (\ref{afourier}) at $\lambda=\mu_n$ yields
\begin{equation}
  a(\mu_n) = \frac{-1}{4n\pi} \int_{-1}^1 \sin(n\pi y) q_\mathrm{o}\big(\textstyle\frac{y+1}{2}\big) dy + O(n^{-2}).
\label{eq:aFou}
\end{equation}
The integrals in (\ref{eq:aFou}) form the sequence of (sine) Fourier coefficients of a square-integrable function, which
therefore is a square-summable sequence; this places the sequence $(a(\mu_n))_{n\in\mathbb{N}}$ in~$\ell^2_1$.
\end{proof}

\noindent
The asymmetry function has the following additional properties, which are proved in \cite[Theorem\,2]{Shipman2019}.

\begin{spslist}
  \item  The potential $q$ is symmetric if and only if $a(\lambda;q)\equiv0$.
  \item  The DtN maps $N(\lambda;q_1)$ and $N(\lambda;q_2)$ commute if and only if $a(\lambda;q_1)=a(\lambda;q_2)$.
  \item  The Dirichlet spectrum of $-d^2/dx^2+q(x)$, together with $a(\lambda;q)$, determine $q\in L^2[0,1]$ uniquely.
  \item
  $
 \displaystyle
  c'(1,\lambda;q)\,a(\lambda;q) \;=\;  -\!\int_0^1 \! q_\mathrm{o}(x)\,c(x,\lambda;q)\,  c(x,\lambda;\tilde q)\,dx\,,
$
\,where $q_\mathrm{o}(x) = \frac{1}{2}\big(q(x)-q(1-x)\big)$.
\end{spslist}

\noindent
Property (3) will emerge as a result of the main Theorem~\ref{thm:main}.

\section{An interpolation theorem}

The question of whether a complex function from a certain class can be uniquely determined from its values on a given discrete set of points is fundamental in interpolation and sampling theory, going all the way back to Shannon's famous sampling theorem \cite{Shannon}
 based on work of Whittaker \cite{Whittaker}. For functions sampled on the whole complex plane, there are many results relating the growth of the function to the required density of the point set, see, e.g.~\cite{Lindholm,Marco}. For functions sampled only at points on the real
line, the classical result on the required density of the discrete set is \cite[Theorem 3.3]{Levin}. In our case, we will be sampling functions in the class $\Ehalf$ on the Dirichlet spectrum of a Sturm-Liouville operator. This set does not satisfy the density assumptions of any of the results mentioned above.  However, the class $\Ehalf$ involves an additional decay assumption on the entire functions being sampled which will enable us to prove the desired result. Theorem~\ref{thm:interpolation} is the key result in this section, as it tells us that the asymmetry function is completely determined by its values at the Dirichlet spectrum of a potential $q\in L^2[0,1]$.
We give a full proof, which adapts arguments of McKean and Trubowitz~\cite[\S5]{McKeanTrubowitz1976a}.

The entire function
\begin{equation}\label{e}
  e(\omega) := s'(1,\omega^2) - i\omega s(1,\omega^2)
\end{equation}
is a de\,Branges function since it satisfies the inequality in the following lemma.
See \cite[Ch\,2\,\S19]{deBranges1968} for the theory of de\,Branges spaces.

\begin{lemma}\label{lemma:debranges}
For all $\omega\in\RR$,
\begin{equation}\label{easymp}
  |e(\omega)|^2=1+o(1)
  \qquad (|\omega|\to\infty),
\end{equation}
and if $e(\omega)=0$, then $\omega=0$ and $s'(1,0)=0$.

If the least Dirichlet eigenvalue $\mu_1$ of $-d^2/dx^2+q(x)$ on $[0,1]$ is positive, then for all $\omega\in\CC$ with $\Im\,\omega>0$,
\begin{equation}
  |e(\omega)| \;>\; |e(\bar\omega)|.
\end{equation}
\end{lemma}

\begin{proof}
According to the estimates in \cite[\S1.2]{LevitanSargsjan1991a} or \cite[\S1 Theorem\,3]{PoschelTrubowitz1987}, for $\omega\in\RR$,
\begin{equation}
  |e(\omega)|^2\;=\;s'(1,\omega^2)^2+\omega^2s(1,\omega^2)^2\;=\;1+o(1) \qquad (|\omega|\rightarrow\infty).
\end{equation}
Since $s(1,\lambda)$ and $s'(1,\lambda)$ are real when $\lambda$ is real, $e(\omega)=0$ for $\omega\in\RR$ implies that $s'(1,\omega^2)=\omega s(1,\omega^2)=0$.  Since $s$ and $s'$ cannot simultaneously vanish, $\omega=0$, and thus $s'(1,0)=0$.

The proof of the inequality is a modification of the proof of \cite[\S5\,Lemma\,1]{McKeanTrubowitz1976a}.  Since $s(x,\lambda)$ is real for $\lambda\in\RR$, we have $s(x,\bar\lambda)=\overline{ s(x,\lambda)}$.  For $\Im\,\omega>0$,
\begin{align}
  |e(\omega)|^2 - |e(\bar\omega)|^2 &\;=\; 4\Im\left( \omega\, s(1,\omega^2) s'(1,\bar\omega^2) \right)  \\
   &\;=\;  4\Im\,\omega  \int_0^1[\,s(x,\omega^2) s'(x,\bar\omega^2)\,]' dx \\
   &\;>\;  4\Im\,\omega \int_0^1 (|s'(x,\omega^2)|^2+q(x)|s(x,\omega^2)|^2) dx \\
   &\;\geq\;  (4\Im\,\omega) \mu_1\left\| s(\cdot,\omega^2) \right\|^2_{L^2} .
\end{align}
The last inequality comes from the Rayleigh quotient inequality for the quadratic form of the Dirichlet operator $-d^2/dx^2+q(x)$.  
\end{proof}

To the function $e(\omega)$ is associated a de Branges space $B=B_q$, which is a reproducing-kernel Hilbert space.  
It consists of all entire functions $f$ such that
\begin{equation}
  \int_\RR \left| \frac{f(\lambda)}{e(\lambda)} \right|^2 \!d\lambda \;<\; \infty
\end{equation}
and there exists a real number $C_{\!f}$ such that, for all $\omega$ with $\Im\,\omega>0$,
\begin{equation}
  \left| \frac{f(\omega)}{e(\omega)} \right|,\; \left| \frac{\overline{f(\bar\omega)}}{e(\omega)} \right|
  < \frac{C_{\!f}}{\sqrt{\Im\,\omega\,}}.
\end{equation}
The inner product $B[ \cdot,\,\cdot ]=B_q[ \cdot,\,\cdot ]$ in $B$ is that of $L^2(\RR,d\lambda/|e(\lambda)|^2)$; that is, through restriction of functions in $B$ to $\RR$, $B$ is identified with a closed subspace of this weighted $L^2$ space.  The reproducing kernel is
\begin{equation}\label{eq:one}
  \one_\alpha(\beta)
  \;:=\;  \frac{\overline{e(\alpha)}\,e(\beta) - e(\overline \alpha)\,\overline{e(\overline \beta)}}{2 i (\overline \alpha-\beta)}
  \;=\; \frac{\overline \alpha\, s(1, \overline \alpha^2) s'(1, \beta^2) - \beta\, s'(1, \overline \alpha^2) s(1, \beta^2)}{\overline \alpha - \beta}\,.
\end{equation}
The singularity at $\beta=\bar{\alpha}$ is removable; indeed, $\one_\alpha$ is entire.
Therefore, for all $f\in B$ and all $\omega\in\CC$,
\begin{equation}\label{reproducing}
  f(\omega) \;=\; B[ f,\,\one_\omega ]
  \;=\;  \int_\RR f(\beta) \overline{\one_\omega(\beta)}\, \frac{d\beta}{|e(\beta)|^2}\,.
\end{equation}
The function $\one_\alpha(\beta)$ is $K(w,z)$ in \cite[\S19 Theorem\,19]{deBranges1968}.

We assume now, by adding a suitable constant to the potential, that
\begin{align}
 s'(1,0) &\not = 0, \quad s(1,0) \not = 0.
\label{edirneu}\end{align}
Let $(\mu_k)_{k\in\mathbb{N}}$ denote the sequence of Dirichlet eigenvalues associated to the operator  $-d^2/dx^2+q(x)$ on $[0,1]$.
We then have
\begin{align}\label{oneateig}
 \one_\omega(\pm \sqrt{\mu_j}) &= \frac{\overline \omega\, s(1, \overline \omega^2) s'(1, \mu_j)}{\overline \omega \mp \sqrt{\mu_j}},
\end{align}
so if
$\beta = \pm \sqrt{\mu_j}$
and
$\omega \in \{\pm \sqrt{\mu_k} : k \in {\mathbb N}\} \setminus \{\beta\},$
then
$\one_\omega(\beta) = 0;$
and, as
$\overline \omega \rightarrow \pm \sqrt{\mu_j}$\,,
\begin{align}
 \one_\omega(\pm \sqrt{\mu_j}) &\rightarrow \pm \sqrt{\mu_j}\, s'(1, \mu_j) \,\frac \partial{\partial \omega} s(1, \omega^2) \Big|_{\omega = \pm \sqrt{\mu_j}}
\nonumber\\
 &= 2 \mu_j\,s'(1, \mu_j)\,\frac \partial{\partial \lambda}s(1, \mu_j) = 2 \mu_j l_j^2,
\label{enorm}\end{align}
where we have used $s'(1, \mu_j)\,\frac \partial{\partial \lambda}s(1, \mu_j) = l_j^2$~\cite[\S2 Theorem 2]{PoschelTrubowitz1987} with
\begin{align}
 l_j^2 &= \int_0^1 s(x, \mu_j)^2\,d x \qquad (j \in {\mathbb N})
\nonumber\end{align}
being the Dirichlet-Dirichlet norming constants. 
For suitable functions $f$ and $g$, define
\begin{align}
 A[f,g] = A_q[f,g] \;:=\; \frac{f(0)\,\overline{g(0)}}{s'(1,0)\,s(1,0)} + \sum_{j=1}^\infty \frac{1}{2 \mu_j l_j^2}\,\big(f(-\sqrt{\mu_j}) \overline{g(-\sqrt{\mu_j})} + f(\sqrt{\mu_j}) \overline{g(\sqrt{\mu_j})}\big).
\nonumber\end{align}
Define the space $A=A_q$ to consist of all functions $f$ defined on $\{0\}\cup\{\pm\sqrt{\mu_j}\}_{j\in\NN}$ such that
\begin{equation}\label{A}
  A[f,f]<\infty.
\end{equation}

The following two lemmas comprise an analog of \cite[\S5\,Lemma\,2]{McKeanTrubowitz1976a} for 
Dirichlet eigenvalue sequences.

\begin{lemma}\label{lemma:AB}
Let the potential $q\in L^2[0,1]$ satisfy $\mu_1>0$, $s'(1,0)\not=0$, and $s(1,0)\not=0$.
Then for all $\alpha$ and $\beta$ in $\CC$,
\begin{equation}
  A_q[\one_\alpha, \one_\beta] \;=\; B_q[\one_\alpha, \one_\beta].
\end{equation}
\end{lemma}

\begin{proof}
From (\ref{oneateig}) it follows for $\alpha,\beta\neq\pm\sqrt{\mu_j}$ that
\begin{align}
 \one_\alpha(\pm \sqrt{\mu_j}) \overline{\one_\beta(\pm \sqrt{\mu_j})} &= \frac{\overline \alpha s(1,\overline \alpha^2)s'(1,\mu_j)}{\overline \alpha \mp \sqrt{\mu_j}}\,\frac{\beta s(1,\beta^2) s'(1,\mu_j)}{\beta \mp \sqrt{\mu_j}}
\nonumber\\
 &= \frac{\overline \alpha \beta\,s(1,\overline \alpha^2)\,s(1,\beta^2)\,s'(1,\mu_j)^2}{(\pm \sqrt{\mu_j} - \overline \alpha)(\pm \sqrt{\mu_j} - \beta)},
\nonumber\end{align}
and further
\begin{align}
 &\one_\alpha(\sqrt{\mu_j}) \overline{\one_\beta(\sqrt{\mu_j})} + \one_\alpha(-\sqrt{\mu_j}) \overline{\one_\beta(-\sqrt{\mu_j})}
\nonumber\\
 &= \overline \alpha \beta s(1, \overline \alpha^2) s(1, \beta^2) s'(1, \mu_j)^2 \left(\frac{1}{(\sqrt{\mu_j}-\overline \alpha)(\sqrt{\mu_j} - \beta)} + \frac{1}{(-\sqrt{\mu_j}-\overline \alpha)(-\sqrt{\mu_j} - \beta)} \right)
\nonumber\\
 &= 2 \overline \alpha \beta\,s(1, \overline \alpha^2)\,s(1, \beta^2)\,s'(1, \mu_j)^2\,\frac{\mu_j + \overline \alpha \beta}{(\mu_j - \overline \alpha^2)(\mu_j - \beta^2)}.
\nonumber\end{align}
Also, from \eqref{eq:one}
we find
\begin{align}
 \one_\alpha(0)\,\overline{\one_\beta(0)} &= s(1,\overline{\alpha}^2)\,s(1, \beta^2)\,s'(1, 0)^2.
\nonumber\end{align}
This yields, by partial fraction decomposition,
\begin{align}
 A[\one_\alpha,\one_\beta] \;&=\; \frac{s'(1, 0)}{s(1,0)}\,s(1,\overline{\alpha}^2) s(1,\beta^2) + \sum_{j=1}^\infty \frac{\overline \alpha \beta}{l_j^2}\,s(1, \overline \alpha^2)\,s(1, \beta^2)\,s'(1,\mu_j)^2\,\frac{\mu_j + \overline \alpha \beta}{\mu_j\,(\mu_j - \overline \alpha^2)\,(\mu_j - \beta^2)}
\nonumber\\
 &=\; \frac{s'(1, 0)}{s(1,0)}\,s(1,\overline{\alpha}^2) s(1,\beta^2)\;+
\nonumber\\
 &\quad +\; \overline \alpha \beta\,s(1, \overline \alpha^2) \,s(1, \beta^2) \sum_{j=1}^\infty \frac{s'(1,\mu_j)^2}{l_j^2} \left(\frac{1}{\overline \alpha \beta \mu_j} + \frac{1}{\overline \alpha(\overline \alpha - \beta)(\mu_j - \overline \alpha^2)} + \frac{1}{\beta(\beta - \overline \alpha)(\mu_j - \beta^2)}\right).
\nonumber\end{align}

Now we consider the resolvent kernel for the Dirichlet boundary value problem,
\begin{align}
 R(x, y; \omega^2) &= \frac{1} W \cases{u(x, \omega^2) v(y, \omega^2) & \text{ if } y \le x \cr u(y, \omega^2) v(x, \omega^2) & \text{ if } x \le y,}
\nonumber\end{align}
where
$v$ is a solution that
satisfies the boundary condition at 0,
$u$ is a solution that
satisfies the boundary condition at 1
and
$W$
is the Wronskian of
$u, v.$
Here we can take
$v = s$
and
$u = s(1, \omega^2) c - c(1, \omega^2) s.$
Then
\begin{align}
 W &= \left|\nonpmatrix{s(1, \omega^2) c - c(1, \omega^2) s & s \cr s(1, \omega^2) c' - c(1, \omega^2) s' & s'} \right| = s(1, \omega^2),
\nonumber\end{align}
so
\begin{align}
 R(x, y; \omega^2) &= \frac{1}{s(1, \omega^2)} \cases{(s(1, \omega^2) c(x, \omega^2) - c(1, \omega^2) s(x, \omega^2))\,s(y, \omega^2) & \text{ if } y \le x \cr (s(1, \omega^2) c(y, \omega^2) - c(1, \omega^2) s(y, \omega^2))\,s(x, \omega^2) & \text{ if } x \le y.}
\nonumber\end{align}
Now
$R$
may not be differentiable on the diagonal, but off the diagonal we find
\begin{align}
 \frac{\partial^2}{\partial x \partial y}R(x, y; \omega^2) &= \frac{1}{s(1, \omega^2)} \cases{(s(1, \omega^2) c'(x, \omega^2) - c(1, \omega^2) s'(x, \omega^2))\,s'(y, \omega^2) & \text{ if } y < x \cr (s(1, \omega^2) c'(y, \omega^2) - c(1, \omega^2) s'(y, \omega^2))\,s'(x, \omega^2) & \text{ if } x < y,}
\nonumber\end{align}
which extends continuously to the diagonal, giving in particular
\begin{align}
 \frac{\partial^2}{\partial x \partial y}R(1, 1; \omega^2) &= \frac{1}{s(1, \omega^2)}\,\left(s(1, \omega^2) c'(1, \omega^2) - c(1, \omega^2) s'(1, \omega^2)\right)\,s'(1, \omega^2)
= - \frac{s'(1, \omega^2)}{s(1, \omega^2)},
\nonumber\end{align}
using the Wronskian of
$c$ and $s$.
\par
On the other hand, the resolvent kernel can be expressed in terms of the
normalized eigenfunctions as
\begin{equation*}
   R(x, y; \omega^2) \;=\; \sum_{j=0}^\infty \frac{s(x, \mu_j) s(y, \mu_j)}{(\mu_j - \omega^2)\,l_j^2},
\end{equation*}
and differentiating with respect to $x$ and $y$
here and equating the two formulae for the resolvent kernel, we deduce
\begin{equation}
 \sum_{j=0}^\infty \frac{s'(1, \mu_j)^2}{(\mu_j - \omega^2)\,l_j^2} \;=\; - \frac{s'(1, \omega^2)}{s(1, \omega^2)}.
\nonumber
\end{equation}
Hence we can calculate further, from the above partial-fraction decomposition,
\begin{align}
 A[\one_\alpha, \one_\beta] \;&=\; s(1, \overline \alpha^2)\,s(1, \beta^2) \sum_{j=1}^\infty \frac{s'(1, \mu_j)^2}{l_j^2} \left(\frac{1} \mu_j + \frac{\beta}{(\overline \alpha - \beta)(\mu_j - \overline \alpha^2)} + \frac{\overline \alpha}{(\beta - \overline \alpha)(\mu_j - \beta^2)}\right)\;+
\nonumber\\
 &\qquad +\; \frac{s'(1, 0)}{s(1, 0)}\,s(1, \overline \alpha^2)\,s(1, \beta^2)
\nonumber\\
 &=\; s(1, \overline \alpha^2)\,s(1, \beta^2) \left(- \frac{s'(1,0)}{s(1,0)} - \frac{\beta}{\overline \alpha - \beta}\,\frac{s'(1, \overline \alpha^2)}{s(1, \overline \alpha^2)} - \frac{\overline \alpha}{\beta - \overline \alpha}\,{s'(1,\beta^2)}{s(1, \beta^2)} \right)\;+
\nonumber\\
 &\qquad +\; \frac{s'(1, 0)}{s(1, 0)}\,s(1, \overline \alpha^2)\,s(1, \beta^2)
\nonumber\\
 &=\; - \frac{\beta}{\overline \alpha - \beta}\,s(1, \beta^2) s'(1, \overline \alpha^2) - \frac{\overline \alpha}{\beta - \overline \alpha}\,s(1, \overline \alpha^2) s'(1,\beta^2) \\
 &=\; \one_\alpha(\beta) \;=\; B[\one_\alpha,\one_\beta],
\nonumber\end{align}
Since $A[\one_\alpha, \one_\beta]$ and $B[\one_\alpha,\one_\beta]$ are analytic in $\alpha$ and $\beta$,
the equality can be extended to $\alpha,\beta=\pm\sqrt{\mu_j}$.
This completes the proof.
\end{proof}

\begin{lemma}\label{lemma:ABiso}
Let the potential $q\in L^2[0,1]$ satisfy $\mu_1>0$, $s'(1,0)\not=0$, and $s(1,0)\not=0$.  The Hilbert spaces $A_q$ and $B_q$ are isomorphic through the maps
\begin{equation}
  B \to A :: f \mapsto f\big|_{\{0\}\cup\{\pm\sqrt{\mu_j}\}_{j\in\NN}}
\end{equation}
and
\begin{equation}
  A \to B :: g \mapsto f,
  \qquad  f(\omega) = A[g,\one_\omega].
\end{equation}
\end{lemma}

\begin{proof}
We follow the proof of \cite[Lemma\,2]{McKeanTrubowitz1976a}.  Considering Lemma~\ref{lemma:AB}, we only have to show that the functions $\one_\omega$ span $B$ and their restrictions span $A$ in the sense of their linear algebraic spans being dense.
To see that the functions $\one_\omega$ ($\omega\in\CC$) span $B$ observe that, for all $f\in B$, $B[f,\one_\omega]=0$ for all $\omega$ implies that $f=0$ in $B$ because of the reproducing-kernel property (\ref{reproducing}) of~$\one_\omega$.  For $A$, the computations after equation~(\ref{oneateig}) show that the restrictions of the functions $\one_\omega$ to the set $\{0\}\cup\{\sqrt{\mu_j}\}_{j\in\NN}$ for $\omega\in\{0\}\cup\{\sqrt{\mu_j}\}_{j\in\NN}$ span $A$.
\end{proof}

\begin{lemma}\label{lemma:ABspaces}
Let the potential $q\in L^2[0,1]$ satisfy $\mu_1>0$, $s'(1,0)\not=0$, and $s(1,0)\not=0$.
The space $A_q$ with inner product $A[f,g]$ is a Hilbert space that contains the same functions as the classical space $\ell^2(\ZZ)$; and the norms of $A$ and $\ell^2$ are equivalent. The Hilbert space $B_q$ contains the same functions as $\E$; and the norms of $B$ and $L^2(\RR)$ are equivalent.
\end{lemma}

\begin{proof}
Using the estimates in \cite[\S1.2]{LevitanSargsjan1991a}
or \cite[\S1 Theorem\,3]{PoschelTrubowitz1987}, 
\begin{equation}
  l_j^2=\int_0^1 s(x,\mu_j)^2 \;= \int_0^1 \frac{\sin^2(\sqrt{\mu_j}x)}{\mu_j}\ dx +O(\mu_j^{-3/2}) \;=\; \frac{1}{2\mu_j} +O(\mu_j^{-3/2})
\end{equation}
as $j\to\infty$.  Therefore $$0<\frac{1}{\mu_jl_j^2}\;=\;2+O(\mu_j^{-1/2}),$$
and the first statement follows.

We now turn to the claim about $B$ and consider
$$\one_\omega(z)= \frac{\overline{e(\omega)}e(z)-e(\bar{\omega})\overline{e(\bar{z})}}{-2i(z-\bar{\omega})}.$$
The singularity at $z=\bar{\omega}$ is removable, so $\one_\omega$ is entire.
In view of the definition~(\ref{e}) of $e(z)$, since $z s(1,z^2)$ and $s'(1,z^2)$ are of order $1$ and type $1$,
$\one_\omega$ is of order $1$ and type $1$.  Moreover, for $z\in\RR$, the functions $z s(1,z^2)$ and $s'(1,z^2)$ are bounded, so $\one_\omega(z)\simeq(z-\bar\omega)^{-1}\in L^2$. Therefore, $\one_\omega\in\E$.
By the proof of \cite[Lemma 2]{McKeanTrubowitz1976a}, the functions $\one_\omega$ span $B$ and since $\E$ is closed, we have
$$B=\overline{\Span\{\one_\omega\}}\subseteq \E.$$

We next show that $\E\subseteq B$. Lemma \ref{lemma:debranges} shows that the $L^2$-norm is equivalent to $B[f,f]$. As $f\in\E\subseteq L^2$, we have  that $B[f,f]<\infty$. For $f$ to lie in the de Branges space, it remains to show that for $\omega\in\CC^+$ we have
$$\frac{|f(\omega)|^2}{|e(\omega)|^2}\leq \frac{C}{\Im\ \omega}$$
for some constant $C>0$.

For $f\in\E$, by the Paley-Wiener Theorem, $\hat{f}$ is supported in $[-1,1]$, so
$$f(\omega)=\frac{1}{\sqrt{2\pi}}\int_{-1}^1 e^{-i\omega x}\hat{f}(x)\ dx.$$
Thus, setting $\Im\,\omega=b$, by Cauchy-Schwarz,
\begin{equation}\label{eq:estf}
|f(\omega)|^2\leq \frac{1}{2\pi}\int_{-1}^1|\hat{f}(x)|^2\ dx \cdot \int_{-1}^1 e^{2bx}\ dx = O\left(\frac{\sinh{2b}}{b}\right).
\end{equation}

By \cite[Theorem 3]{PoschelTrubowitz1987},
\begin{eqnarray*}
s(1,\omega^2) &= & \frac{\sin(\omega)}{\omega} +O\left(\frac{e^b}{|\omega|^2}\right),\\
\omega s(1,\omega^2) &= & \sin(\omega) +O\left(\frac{e^b}{|\omega|}\right),\\
s'(1,\omega^2) &= & \cos(\omega) +O\left(\frac{e^b}{|\omega|}\right),\\
e(\omega)&=& \cos(\omega)-i\sin(\omega)+O\left(\frac{e^b}{|\omega|}\right) \ =\ e^{-i\omega}\left(1+O\left(\frac{1}{|\omega|}\right)\right).
\end{eqnarray*}
Therefore, 
$$e(\omega)^{-1}=e^{i\omega}\left(1+O\left(\frac{1}{|\omega|}\right)\right),$$
so 
$$ \left|e(\omega)\right|^{-2}=e^{-2b}\left(1+O\left(\frac{1}{|\omega|}\right)\right),$$
and combined with \eqref{eq:estf}, we have that
$$ \left|\frac{f(\omega)}{e(\omega)}\right|^2 = O\left(\frac{\sinh{2b}}{b}e^{-2b}\right)=O\left(\frac{1}{b}\right).$$
Thus, the required estimate is satisfied away from the real axis. By Lemma \ref{lemma:debranges}, $1/|e(\omega)|$ is bounded on the real axis, and since $f$ is square integrable, this implies that the estimate holds on $\CC^+$.
\end{proof}

\begin{theorem}\label{thm:interpolation}
Let $q\in L^2[0,1]$ be given, and let $(\mu_j)_{j\in\NN}$ be the Dirichlet spectral sequence of $-d^2/dx^2+q(x)$ on $[0,1]$.
The restriction of functions in $\Ehalf$ to $(\mu_j)_{j\in\NN}$ is a bounded linear bijection from $\Ehalf$ and $\ell^2_1(\NN)$.
The inverse is effected through the following interpolation formula:  For $\lambda\in\CC$,
\begin{equation}
  \phi(\lambda) \;=\; \sum_{j=1}^\infty \phi(\mu_j) \prod_{i\not=j} \frac{\mu_i-\lambda}{\mu_i-\mu_j}.
\end{equation}
\end{theorem}

\begin{proof}
Since the space $\Ehalf$ is invariant under shifts $\phi(\lambda)\mapsto\phi(\lambda+c)$, with $c\in\RR$, 
it suffices to prove the theorem for potentials $q$ such that $\mu_1>0$, $s'(1,0)\not=0$, and $s(1,0)\not=0$, as this can be accomplished by shifting the potential by a constant.
Let $q$ satisfy these three spectral properties for the rest of the proof.

The transformation
\begin{equation}\label{eq:half-odd}
  \phi \mapsto f,
  \qquad
  f(\omega) = -\omega\phi(\omega^2)
\end{equation}
is a bijection from the space of entire functions of order $\leq1/2$ and type $\leq1$ to the space of odd entire functions of order $\leq1$ and type $\leq1$.  It maps $\Ehalf$ onto the space of odd functions in $\E$ because
\begin{equation}
  \int_\RR |f(\omega)|^2d\omega
  \;=\; 2\int_0^\infty |\omega\phi(\omega^2)|^2 d\omega
  \;=\; \int_0^\infty |\phi(\lambda)|^2\lambda^{1/2} d\lambda.
\end{equation}

Now consider 
\begin{equation}\label{eq:odd-one}
(a_j)_{j\in\ZZ} \;\mapsto\; (-\mu_j^{-1/2} a_j )_{j\in\NN}
\end{equation}
where $a_{-j} = - a_j$;
this is an isomorphism from the odd subspace of $\ell^2(\ZZ)$ to $\ell^2_1(\NN)$ because $\sqrt{\mu_j}\sim \pi j$ as $j\to\infty$. Therefore, we have the following succession of bounded linear bijections: From $\Ehalf$ to the odd functions in $\E$ given by \eqref{eq:half-odd}, then to odd functions in $B$ by Lemma \ref{lemma:ABspaces}, to odd sequences in $A$ by Lemma \ref{lemma:ABiso}, to odd sequences in $\ell^2(\ZZ)$ by Lemma \ref{lemma:ABspaces}, and finally to $\ell^2_1(\NN)$ by  \eqref{eq:odd-one}. This amounts to
 restriction of functions in $\Ehalf$ to $(\mu_j)_{j\in\NN}$, as
$$\phi \mapsto -\omega \phi(\omega^2)\mapsto -\omega \phi(\omega^2)\mapsto \big(\sqrt{\mu_j}\phi(\mu_j),-\sqrt{\mu_j}\phi(\mu_j)\big) \mapsto \big(\sqrt{\mu_j}\phi(\mu_j),-\sqrt{\mu_j}\phi(\mu_j)\big) \mapsto (\phi(\mu_j)),$$ 
 proving the first statement of the theorem.

To prove the interpolation, use Lemma~\ref{lemma:AB}, noting that $f(0)=0$, for $\omega\in\RR$,
\begin{align}
  -\omega\phi(\omega^2) \;=\; f(\omega) \;&=\; B[f,\one_\omega] \;=\; A[f,\one_\omega] \\
   &=\; \sum_{j=1}^\infty \frac{1}{2\mu_j\ell_j^2}
                   \left( f(-\sqrt{\mu_j}) \overline{\one_\omega(-\sqrt{\mu_j})}
                           + f(\sqrt{\mu_j}) \overline{\one_\omega(\sqrt{\mu_j})} \right) \\
   &=\; \sum_{j=1}^\infty \frac{1}{2\mu_j\ell_j^2}
   \left( f(-\sqrt{\mu_j})\frac{\omega\,s(1,\omega^2)s'(1,\mu_j)}{\omega+\sqrt{\mu_j}}
           +  f(\sqrt{\mu_j})\frac{\omega\,s(1,\omega^2)s'(1,\mu_j)}{\omega-\sqrt{\mu_j}} \right) \\
   &=\; \sum_{j=1}^\infty \frac{1}{2\mu_j\ell_j^2}
            \sqrt{\mu_j} \phi(\mu_j) \left( \omega s(1,\omega^2)s'(1,\mu_j) \right)
            \left( \frac{1}{\omega+\sqrt{\mu_j}} - \frac{1}{\omega-\sqrt{\mu_j}} \right) \\
   &=\; \sum_{j=1}^\infty \frac{1}{\ell_j^2}
             \frac{\phi(\mu_j)}{\mu_j-\omega^2}
             \left( \omega\, s(1,\omega^2)s'(1,\mu_j) \right) \\
   &=\; \omega \sum_{j=1}^\infty \phi(\mu_j) \frac{s(1,\omega^2)}{\dot s(1,\mu_j)(\mu_j-\omega^2)} \\
   &=\; -\omega \sum_{j=1}^\infty \phi(\mu_j) \prod_{i\not=j} \frac{\mu_i-\omega^2}{\mu_i-\mu_j}.
\end{align}
We have used the identity
\begin{equation}
  \ell_j^2 = s'(1,\mu_j)\dot s(1,\mu_j),
\end{equation}
where $\dot s = ds/d\lambda$, and the representation of the entire function $s(1,\lambda)$ of order $1/2$  in terms of its roots, given by the Hadamard factorization.
The formula for $\phi(\lambda)$ in the theorem follows.
\end{proof}

\section{Asymmetry classes in inverse spectral theory}

The goals of this section are (1) to establish a bijection between square integrable potentials $q$ and pairs $((\mu_n)_{n\in\NN},a)$ of spectral sequences and asymmetry functions; and (2) to establish the analyticity of this correspondence.  The first part rides on an inverse spectral theory of P\"oschel and Trubowitz~\cite{PoschelTrubowitz1987} for the Dirichlet Schr\"odinger operator on an interval.

\subsection{Bijective correspondence}\label{sec:bijection}

Let $\mu(q)=(\mu_n(q))_{n\in\NN}$ be the Dirichlet eigenvalue sequence for a potential $q\in L^2[0,1]$.  The image of $\mu$ on $L^2$ consists of all real strictly increasing sequences $(\mu_n)_{n\in\NN}$ of the form
\begin{equation}\label{sigma}
  \mu_n \;=\; \pi^2n^2 + c + \sigma_n,
\end{equation}
in which $c\in\RR$ and $(\sigma_n)_{n\in\NN}$ is in $\ell^2$.
The set of sequences $(\sigma_n)_{n\in\NN}$ such that (\ref{sigma}) is strictly increasing is an open set in $\ell^2$, and it is denoted by $\mathcal{S}$.
The spectral data introduced in~\cite{PoschelTrubowitz1987} are the constants 
\begin{equation}\label{kappa}
  \kappa_n(q) \;:=\; \log|s'(\mu_n(q),1;q)| \;=\; \log( (-1)^n s'(\mu_n(q),1;q)).
\end{equation}
The sequence $\kappa(q)=(\kappa_n(q))_{n\in\NN}$ lies in the space $\ell^2_1$.
It is proved in~\cite{PoschelTrubowitz1987} that the correspondence
\begin{equation}
  L^2 \to \RR\times\mathcal{S}\times\ell^2_1 \;::\;
  q(x) \mapsto \big(c,\,(\sigma_n)_{n\in\NN},\, (\kappa_n)_{n\in\NN} \big)
\end{equation}
is bijective and analytic.

By considering the Wronskian, we see that $c(\mu_n(q);q)s'(\mu_n(q);q)=1$. Therefore, the sequence $(\kappa_n(q))_{n\in\NN}$ is related to $a(\lambda;q)$ by evaluation at the Dirichlet spectrum of~$q$,
\begin{equation}\label{alphakappa}
\begin{split}
  \alpha_n(q) \;:=\; a(\mu_n(q);q) &\;=\; \onehalf\big( c(\mu_n(q); q) - s'(\mu_n(q);q) \big)\\
     &\;=\; \onehalf\big( s'(\mu_n(q);q)^{-1} - s'(\mu_n(q);q) \big)\\
     &\;=\; (-1)^n \onehalf \big( e^{-\kappa_n(q)} - e^{\kappa_n(q)} \big)\\
      &\;=\; (-1)^{n+1} \sinh \kappa_n(q).
\end{split}
\end{equation}
The correspondence $\kappa_n(q) \leftrightarrow \alpha_n(q)$ is a bijection of $\ell^2_1$, and therefore the data $(\alpha_n(q))_{n\in\NN}$ can be used instead of $(\kappa_n(q))_{n\in\NN}$, that is,
the correspondence
\begin{equation}
  L^2 \to \RR\times\mathcal{S}\times\ell^2_1 \;::\;
  q(x) \mapsto \big(c,\,(\sigma_n)_{n\in\NN},\, (\alpha_n)_{n\in\NN} \big)
\end{equation}
is bijective.

Let $p\in L^2[0,1]$ be given, and denote by $M(p)$ the set of potentials isospectral to~$p$,
\begin{equation}
  M(p) := \left\{ q\in L^2[0,1] : \mu(q) = \mu(p) \right\}.
\end{equation}
Proposition~\ref{prop:a}, the interpolation theorem, and the inverse spectral theory provide the following diagram.
\begin{equation}
\renewcommand{\arraystretch}{1.1}
\left.
\begin{array}{ccccccc}
  M(p) &\longrightarrow& \Ehalf &\longrightarrow& \ell^2_1 &\longrightarrow& M(p)  \\
  q &\longmapsto& a &\longmapsto& (a(\mu_n(p)))_{n\in\NN} &\longmapsto& q   
\end{array}
\right.
\end{equation}
The restriction map $\Ehalf\to\ell^2_1$ is a bijection by Theorem~\ref{thm:interpolation}.  The composite map from $M(p)\to M(p)$ is the identity map~\cite[Ch.3,Theorems\,4,5]{PoschelTrubowitz1987}; and the composite map $M(p)\to\ell^2_1$ is surjective~\cite[Ch.4,Theorems\,1*,3]{PoschelTrubowitz1987}.
This makes the map $M(p)\to\Ehalf$ a bijection, resulting in the following theorem.

\begin{theorem}
  The set of asymmetry functions $a(\lambda;q)$, as $q$ runs over $L^2[0,1]$, is equal to $\Ehalf$; and for each isospectral set $M(p)$, the correspondence $q\mapsto a(\cdot\,;q)$ is a bijection between $M(p)$ and~$\Ehalf$.
\end{theorem}

As a corollary, we obtain the main result of the paper, as announced in the abstract.

\begin{theorem}\label{thm:main}
The set of asymmetry functions $a(\lambda;q)$, as $q$ runs over $L^2[0,1]$, is equal to $\Ehalf$.  Let $a\in\Ehalf$ be given.  The set of potentials possessing $a(\lambda)$ as its asymmetry function consists of one $q\in L^2[0,1]$ for each Dirichlet spectral sequence.
\end{theorem}

\subsection{Analyticity}

We now establish the bi-analyticity of the map between potentials $q\in L^2[0,1]$ and spectral data
$(c,(\sigma_n)_{n\in\NN},a)$ $\in$ $\RR\times\mathcal{S}\times\Ehalf$.

\begin{theorem}\label{thm:analytic}
The map
\begin{equation*}
  L^2[0,1] \;\longrightarrow\; \RR\times\mathcal{S}\times\Ehalf  \;\;::\;\;
  q \;\longrightarrow\ \big(c,\,(\sigma_n)_{n\in\NN},\, a \big),
\end{equation*}
from potentials $q$ to the spectrum $(\mu_n \!=\! c + \sigma_n)_{n\in\NN}$ and asymmetry function $a$ of the Dirichlet Schr\"odinger operator $-d^2/dx^2+q(x)$ on $[0,1]$ is bi-analytic.  ($\mathcal{S}$ is defined before~(\ref{kappa}).)
\end{theorem}

This will be proved through the following maps:
\begin{equation}
\renewcommand{\arraystretch}{1.1}
\left.
\begin{array}{ccccccc}
  L^2[0,1] &\longleftrightarrow& \RR\times\mathcal{S}\times\ell^2_1 &\longleftrightarrow& \RR\times\mathcal{S}\times\ell^2_1 &\longleftrightarrow& \RR\times\mathcal{S}\times\Ehalf  \\
  q &\longleftrightarrow& \big(c,\,(\sigma_n)_{n\in\NN},\, (\kappa_n)_{n\in\NN} \big) &\longleftrightarrow& \big(c,\,(\sigma_n)_{n\in\NN},\, (\alpha_n)_{n\in\NN} \big) &\longleftrightarrow& \big(c,\,(\sigma_n)_{n\in\NN},\, a \big)
\end{array}
\right.
\end{equation}
Analyticity of the map given by the first arrow is due to P\"oschel and Trubowitz~\cite{PoschelTrubowitz1987}.
The analyticity implied by the second arrow is very easy because $\alpha_n = (-1)^{n+1} \sinh \kappa_n$.
The rest of this section establishes analyticity of the map given by the last arrow from right to left.
Because of the inverse function theorem for analytic functions in Banach spaces (see \cite[p.\,1081]{Whittlesey1965} or \cite[p.\,142]{PoschelTrubowitz1987}), analyticity in one direction implies analyticity in the other.

\smallskip

Theorem~\ref{thm:interpolation} provides an identification of $\Ehalf$ with $\ell^2_1$ through evaluation on the sequence $(\pi^2n^2)_{n\in\NN}$.  This is a bounded linear bijection of Hilbert spaces and is therefore analytic.  Denote the values of $a(\lambda)$ on this sequence by $(a_n)_{n\in\NN}$,
\begin{equation}
  a_n \;:=\; a(\pi^2n^2).
\end{equation}
Thus the analyticity of the map given by the third arrow pointing leftward is established by proving that the map
\begin{equation}\label{leftarrow}
  \big(c, (\sigma_n)_{n\in\NN}, (a_n)_{n\in\NN}\big) \;\;\mapsto\;\; \big(c, (\sigma_n)_{n\in\NN},  (a(n^2\pi^2+c+\sigma_n))_{n\in\NN} \big)
\end{equation}
from $\RR\times\mathcal{S}\times\ell^2_1$ to $\RR\times\mathcal{S}\times\ell^2_1$ is analytic.

By \cite[Theorem\,3,\,p.138]{PoschelTrubowitz1987}, analyticity of a map between an open subset of a complex Banach space to another Hilbert space is equivalent to the map satisfying two properties simultaneously, (1) weak analyticity of each projection to the elements of an orthonormal basis of the target space, and (2) local boundedness.
We therefore work with the complexification of $\RR\times\ell^2\times\ell^2_1$, namely
$(\RR\times\ell^2\times\ell^2_1)\otimes\CC = (\RR\times\ell^2\times\ell^2_1)+i(\RR\times\ell^2\times\ell^2_1)$.  We care only about analyticity on the open subset $\RR\times\mathcal{S}\times\ell^2_1$ of the real subspace of this complex Banach space.  This open subset is contained in an open subset $\CC\times\mathcal{S}_\CC\times(\ell^2_1\tensor\CC)$ of $(\RR\times\ell^2\times\ell^2_1)\otimes\CC$, in which
\begin{equation}
  \mathcal{S}_\CC \;=\; \bigcup\limits_{\sigma\in\mathcal{S}} U_\sigma,
\end{equation}
with $U_\sigma$ being an open neighborhood of $\sigma$ in $\ell^2\tensor\CC$ with $U_\sigma\cap\ell^2\subset\mathcal{S}$ so that
\begin{equation}
  \mathcal{S}_\CC \cap \ell^2 \;=\; \mathcal{S}.
\end{equation}

Weak analyticity of each coordinate is due to the calculation
\begin{multline}
  \frac{d}{dh}
  \left( 
    c+h\hat c, (\sigma_n+h\hat\sigma_n)_{n\in\NN}, ((a+h\hat a)(n^2\pi^2 + c + h\hat c + \sigma_n + h\hat\sigma_n))_{n\in\NN} \right) \\
    = \left( \hat c, (\hat \sigma_n)_{n\in\NN}, ( (\hat c + \hat\sigma_n) a'(n^2\pi^2+c+\sigma_n) + \hat a(n^2\pi^2 + c + \sigma_n )  )_{n\in\NN}
  \right).
\end{multline}
Now we need to prove local boundedness, which is the content of Proposition~\ref{prop:atoalpha}.
By applying Theorem~\ref{thm:interpolation} to $\Ehalf$ and $i\Ehalf$, an asymmetry function $a\in\EhalfC$ is recovered from $(a_n)_{n\in\NN}$ through interpolation,
\begin{equation}
  a(\lambda) \;=\; \sum_{j=1}^\infty a_j \prod_{k\not=j}  \frac{\pi^2k^2-\lambda}{\pi^2k^2-\pi^2j^2}
  \qquad
  \big(a_j = a(\pi^2j^2)\big).
\end{equation}
Therefore, we expect the formula
\begin{equation}\label{alphan}
  \alpha_n \;=\; a(\pi^2n^2+c+\sigma_n) \;=\;  \sum_{j=1}^\infty a_j \prod_{k\not=j} \frac{\pi^2k^2-\pi^2n^2-c_n}{\pi^2k^2-\pi^2j^2},
\end{equation}
in which $c_n=c+\sigma_n$, to extend the map (\ref{leftarrow}) to one from
$\CC\times\mathcal{S}_\CC\times(\ell^2_1\tensor\CC)$ to itself.

\begin{proposition}\label{prop:atoalpha}
  The map
\begin{equation}
  \left(c, (\sigma_n)_{n\in\NN}, (a_n)_{n\in\NN} \right)
  \quad\mapsto\quad
  \left(c, (\sigma_n)_{n\in\NN}, (\alpha_n)_{n\in\NN} \right),
\end{equation}
with $\alpha_n$ defined in (\ref{alphan}) is a locally bounded map from
$\CC\times\mathcal{S}_\CC\times(\ell^2_1\tensor\CC)$ to itself.
\end{proposition}

The proof of the proposition rests on two lemmas.
The definition of $\alpha_n$ can be expressed as
\begin{equation}
  n\alpha_n \;=\; \sum_{j=1}^\infty j a_j K(n,j),
\end{equation}
in which
\begin{equation}
  K(n,j) \;=\; \frac{n}{j} \prod_{k\not=j} \frac{\pi^2k^2-\pi^2n^2-c_n}{\pi^2k^2-\pi^2j^2}
  \;=\;
  \renewcommand{\arraystretch}{1.1}
\left\{
\begin{array}{ll}
  \displaystyle
  \prod_{k\not=n}\Big( 1 - \frac{c_n}{\pi^2k^2-\pi^2n^2} \Big), & n=j \\
  \displaystyle
  \frac{n}{j} \cdot\frac{-c_n}{\pi^2n^2-\pi^2j^2} \prod_{k\not\in\{j,n\}} \frac{\pi^2k^2-\pi^2n^2-c_k}{\pi^2k^2-\pi^2j^2}, & n\not=j.
\end{array}
\right.
\end{equation}

\begin{lemma}\label{lemma:K}
If $j^2\not=n^2+c_n/\pi^2$, then
\begin{align}
    K(n,j)
    &\;=\; (-1)^{j+1} \frac{2nj}{j^2-n^2-c_n/\pi^2\,} \cdot \mathrm{sinc} \,\pi\sqrt{n^2+c_n/\pi^2} \\
    &\;=\; (-1)^{n+j+1} \frac{j^2}{j^2-n^2-c_n/\pi^2\,} \cdot \frac{2 e_n}{\pi\,j\sqrt{n^2+c_n/\pi^2\,}},
\end{align}
in which
\begin{equation}
  e_n = d_n\mathop{\mathrm{sinc}}\frac{d_n}{n},
  \qquad
  d_n = \frac{c_n/\pi}{1+\sqrt{1+\frac{c_n}{\pi^2 n^2}}},
\end{equation}
and the argument of the square root is taken to lie in $(-\pi,\pi]$.

If $n^2+c_n/\pi^2=j^2$ with $j>0$, then
\begin{equation}
  K(n,j) \;=\; \frac{n}{j},
\end{equation}
and if $n^2+c_n/\pi^2=0$,
\begin{equation}\label{eq:0}
  K(n,j) \;=\; (-1)^{j+1}\frac{2n}{j}.
\end{equation}
\end{lemma}

\begin{remark}
\begin{enumerate}
	\item   $K(n,j)$ is an entire function of $c_n$ as a complex variable.  In the first expression of the proposition, $c_n=\pi^2(j^2-n^2)$ is a removable singularity because the $\mathrm{sinc}$ function vanishes there.
	\item Note that we have $e_n\sim c_n/\pi$ as $n\to\infty$, while when $n^2+c_n/\pi^2=0$, we have $d_n=-n^2\pi$, and thus $e_n=0$.
\end{enumerate}
\end{remark}

\begin{proof}
To prove Lemma~\ref{lemma:K}, write $K(n,j)$ as
\begin{equation}
  K(n,j) \;=\; \frac{n}{j}
  \prod_{k\not=j}\left( 1-\frac{n^2+\gamma_n}{k^2} \right)
  \prod_{k\not=j}\left( 1-\frac{j^2}{k^2} \right)^{\!\!-1}
\end{equation}
with $\gamma_n=c_n/\pi^2$ 
and use the formula
\begin{equation}
  \prod_{k=1}^\infty \left( 1-\frac{z^2}{k^2} \right) \;=\; \frac{\sin \pi z}{\pi z} \;=\; \mathrm{sinc}\,\pi z.
\end{equation}
Excluding $k=j$ from the product yields
\begin{equation}
  \prod_{k\not=j} \left( 1-\frac{z^2}{k^2} \right) \;=\; \left( 1-\frac{z^2}{j^2} \right)^{\!\!-1} \frac{\sin \pi z}{\pi z}
  \;=\; \frac{j^2\sin \pi z}{\pi z(j^2-z^2)},
\end{equation}
and thus
\begin{equation}
   \prod_{k\not=j} \left( 1-\frac{j^2}{k^2} \right) \;=\; 
   \lim_{z\to j} \frac{j^2\sin \pi z}{\pi z(j^2-z^2)}\
   \;=\; j^2\lim_{z\to j} \frac{\cos \pi z}{j^2-3z^2} \;=\; \frac{(-1)^{j+1}}{2},
\end{equation}
and we obtain, for the second infinite product in $K(n,j)$,
\begin{equation}\label{prod2}
  \prod_{k\not=j} \left( 1-\frac{j^2}{k^2} \right)^{\!\!-1} \;=\; 2 (-1)^{j+1}.
\end{equation}
If $j^2=n^2+c_n/\pi^2$, then for the first product, the foregoing calculation gives
\begin{equation}
  \prod_{k\not=j}\left( 1-\frac{n^2+\gamma_n}{k^2} \right) \;=\; \frac{(-1)^{j+1}}{2},
\end{equation}
and putting this together with (\ref{prod2}) yields the result.
If $j^2\not=n^2+c_n/\pi^2$,
\begin{equation}
  \prod_{k\not=j}\left( 1-\frac{n^2+\gamma_n}{k^2} \right)
  \;=\; \left( 1-\frac{n^2+\gamma_n}{j^2} \right)^{\!\!-1\,} \prod_{k=1}^\infty \left( 1-\frac{n^2+\gamma_n}{k^2} \right)
  \;=\; \frac{j^2}{j^2-n^2-\gamma_n} \frac{\sin \pi\sqrt{n^2+\gamma_n}}{\pi\sqrt{n^2+\gamma_n}}.
\end{equation}
Using $\sqrt{n^2+\gamma_n} = n+d_n/(n\pi)$, we obtain
\begin{equation}
  \sin \pi\sqrt{n^2+\gamma_n} \;=\; (-1)^n\sin\frac{d_n}{n}
  \;=\; (-1)^n \frac{d_n}{n} \mathrm{sinc}\frac{d_n}{n} \;=\; (-1)^n \frac{e_n}{n}
\end{equation}
and therefore
\begin{equation}\label{prod1}
  \prod_{k\not=j}\left( 1-\frac{n^2+\gamma_n}{k^2} \right) \;=\;\frac{j^2}{j^2-n^2-\gamma_n}
  \frac{(-1)^n e_n}{\,\pi n\sqrt{n^2+\gamma_n}}.
\end{equation}
Putting this together with (\ref{prod2}) yields the result.
\end{proof}
%\phantom{.}\hfill$Q.E.D.$\par\noindent

\begin{lemma}\label{lemma:Kbounds}
  Let $C_1$ and $\{c_n\}_{n\in\NN}$ be such that $|c_n|<C_1$ for all $n\in\NN$.  There is a constant $C$ such that
\begin{equation}
  |K(n,n)| < C
  \qquad
  \hbox{ for all } \,n\in\NN
\end{equation}
and such that
\begin{equation}
  \sum_{n,j\in\NN,\, n\not=j} |K(n,j)|^2 < C.
\end{equation}
\end{lemma}

\medskip
\noindent
\begin{proof}
Write $K(n,j)$ as
\begin{equation}
  K(n,j) \;=\; (-1)^{n+j+1} \frac{2}{\pi} \,\kappa_{nj}\,,
  \qquad
  \kappa_{nj} \;:=\; \frac{j}{j^2-n^2-\gamma_n}\frac{e_n}{\sqrt{n^2+\gamma_n}}.
\end{equation}
When $j=n+r$ ($n\in\NN$, $r\in\NN_0$),
\begin{equation}
  \kappa_{nj} \;=\; \frac{n+r}{r(2n+r)-\gamma_n\,} \frac{e_n}{\sqrt{n^2+\gamma_n}}
    \;:=\; \alpha_{rn}\,,
\end{equation}
and when $n=j+r$ ($j\in\NN$, $r\in\NN_0$),
\begin{equation}
  \kappa_{nj} \;=\; \frac{-j}{r(2j+r)+\gamma_{j+r}\,} \frac{e_{j+r}}{\sqrt{(j+r)^2+\gamma_{j+r}}}
    \;:=\; \beta_{rj}\,,
\end{equation}
and on the diagonal,
\begin{equation}
  \kappa_{nn} \;=\; \alpha_{0n} \;=\; \beta_{0n} \;=\; -\frac{e_n}{\gamma_n} \frac{1}{\sqrt{1+\gamma_n/n^2\,}}\,.
\end{equation}

Now let $|c_n|<C_1$ for all $n\in\NN$, or $|\gamma_n|<\gamma=C_1/\pi^2$.  Let $E$ be such that $|e_n|<E$ for all $n\in\NN$.

First, let's examine $K(n,n)$.  We have
\begin{align}
  \kappa_{nn} &\;=\; - \frac{d_n}{\gamma_n} \,\frac{1}{\sqrt{1+\gamma_n/n^2}}\, \mathrm{sinc}\,\frac{d_n}{n}   \\
    &\;=\; - \frac{\pi}{\,1+\sqrt{1+\gamma_n/n^2\,}} \,\frac{1}{\sqrt{1+\gamma_n/n^2\,}}\, \mathrm{sinc}\,\frac{d_n}{n}.
\end{align}
If $n^2>2C_1/\pi^2$, then $|\gamma_n/n^2|<1/2$ and the $d_n$ lie in a strip $\RR+i(-a,a)$ about the real line, in which the $\mathrm{sinc}$ function is bounded.  So, for some $A>0$, 
\begin{equation}
  |\kappa_{nn}| \;<\; A. 
\end{equation}
From Lemma~\ref{lemma:K}, we have
\begin{equation}
  K(n,n) \;=\; (-1)^n \frac{2n^2}{\gamma_n} \mathrm{sinc}\, \pi\sqrt{n^2+\gamma_n}\,.
\end{equation}
For each $n\in\NN$, $K(n,n)$ is an entire function of $\gamma_n$, for $\mathrm{sinc}\, \pi\sqrt{n^2+\gamma_n}$ vanishes whenever $\gamma_n$ does.  Therefore, there is a constant $C_n$ such that $|K(n,n)|<C_n$ whenever $|\gamma_n|<\gamma$.  Taking $C$ to be the maximum of $2A/\pi$ and the constants $C_n$ over $n^2\leq 2\gamma$, we obtain
\begin{equation}
  |K(n,n)| < C
  \qquad
  \forall\,n\in\NN,
  \quad
  |c_n|<c.
\end{equation}

Next, consider $j=n+r$ with $r\geq1$.  We have
\begin{equation}
  rn\,\alpha_{rn} \;=\; \frac{n+r}{2n+r-\gamma_n/r} \frac{e_n}{\sqrt{1+\gamma_n/n^2}}.
\end{equation}
For all $r\in\NN$ and $n>\gamma+1$,
\begin{equation}
  |rn\,\alpha_{rn}| < \sqrt{2} E,
\end{equation}
and therefore
\begin{equation}
  |K(n,n+r)| < \frac{2\sqrt{2}E}{\pi\,rn}
    \qquad
  (r\in\NN,\,n>\gamma+1).
\end{equation}
Next, write
\begin{equation}
  rK(n,n+r) = (-1)^{n+r+1} \frac{2n(n+r)}{2n+r-\gamma_n/r} \mathrm{sinc}\,\pi\sqrt{n^2+\gamma_n}.
\end{equation}
By analyticity of $\mathrm{sinc}\,\pi\sqrt{n^2+\gamma_n}$ in $\gamma_n$ and the restriction $|\gamma_n|<\gamma$, there exists $C>0$ such that, if $1\leq n\leq 1+\gamma$, then
\begin{equation}
  \left| \mathrm{sinc}\,\pi\sqrt{n^2+\gamma_n} \right| < C.
\end{equation}
Since
\begin{equation}
  \lim_{r\to\infty} \frac{2n(n+r)}{2n+r-\gamma_n/r} \;=\; 2n,
\end{equation}
there exists $r_0\in\NN$ such that, if $r>r_0$ and $n$ is such that $1\leq n\leq\gamma+1$,
\begin{equation}
  \left| \frac{2n(n+r)}{2n+r-\gamma_n/r} \right| \;<\; 3(1+\gamma).
\end{equation}
Thus we obtain
\begin{equation}\label{Kbound1}
  |K(n,n+r)| < \frac{3C_1(1+\gamma)}{r}
      \qquad
  (r>r_0,\,n\leq\gamma+1).
\end{equation}
Because of the analyticity of $K(n,n+r)$ in $\gamma_n$ and the constraint $|\gamma_n|<\gamma$, a bound like (\ref{Kbound1}) (with a different constant) holds also for the finite set of pairs $(n,r)$ with $1\leq r\leq r_0$ and $1\leq n\leq 1+\gamma$, and we obtain
\begin{equation}\label{Kbound2}
  |K(n,n+r)| < \frac{C_2}{r}
      \qquad
  (r\in\NN,\,n\leq\gamma+1).
\end{equation}
For each $r\in\NN$,
\begin{equation}
  \sum_{n\in\NN} |K(n,n+r)|^2 \;<\; \frac{1}{r^2} \left( (\gamma+1)C_2^2 \,+\, C_3^2\!\sum_{n>\gamma+1} \frac{1}{n^2} \right).
\end{equation}
Thus, there is a constant $C$ such that, if $|c_n|<c$ for all $n\in\NN$, then
\begin{equation}\label{L2bound1}
  \sum_{r,n\in\NN} |K(n,n+r)|^2 < C.
\end{equation}

Now consider $n=j+r$ with $r\geq1$.  We have
\begin{equation}
  rj\,\beta_{rj} \;=\; \frac{-j}{2j+r+\gamma_{j+r}/r} \frac{e_{j+r}}{\sqrt{(1+r^2/j^2)^2+\gamma_{j+r}/j^2}}
\end{equation}
Similarly to the previous case,
for all $r\in\NN$ and $j>\gamma+1$, we obtain
\begin{equation}
  |rj\,\beta_{rj}| < \sqrt{2} E,
\end{equation}
and therefore
\begin{equation}
  |K(j+r,j)| < \frac{2\sqrt{2}E}{\pi\,rj}
    \qquad
  (r\in\NN,\,j>\gamma+1).
\end{equation}
Now consider
\begin{equation}
  rK(j+r,j) \;=\; (-1)^j \frac{2(j+r)j}{2j+r+\gamma_{j+r}/r} \mathrm{sinc}\,\pi\sqrt{(j+r)^2+\gamma_{j+r}}\,.
\end{equation}
If $r$ is large enough, then $(j+r)^2+\gamma_{j+r}>0$ for all $j$, so that the $\mathrm{sinc}$ factor is bounded by~$1$ in absolute value.  Then we argue as before and obtain $C$ such that, for all $r\in\NN$ and $j$ such that $1\leq j\leq 1+\gamma$,
\begin{equation}
  \left| \mathrm{sinc}\,\pi\sqrt{(j+r)^2+\gamma_{j+r}} \right| \;<\; C.
\end{equation}
By an identical argument as above, with $n$ replaced by $j$, we obtain 
$r_0\in\NN$ such that, if $r>r_0$ and $n$ is such that $1\leq n\leq\gamma+1$,
\begin{equation}
  \left| \frac{2(j+r)j}{2j+r+\gamma_{j+r}/r} \right| \;<\; 3(1+\gamma),
\end{equation}
and hence
\begin{equation}\label{Kbound3}
  |K(j+r,j)| < \frac{3C_1(1+\gamma)}{r}
      \qquad
  (r>r_0,\,j\leq\gamma+1).
\end{equation}
Similarly to above, this bound is extended to all $r$,
\begin{equation}\label{Kbound4}
  |K(j+r,j)| < \frac{C_2}{r}
      \qquad
  (r\in\NN,\,j\leq\gamma+1).
\end{equation}
In the end, we obtain a constant $C$ such that, if $|c_n|<c$ for all $n\in\NN$, then
\begin{equation}
  \sum_{r,j\in\NN} |K(j+r,j)|^2 < C.
\end{equation}
This, together with (\ref{L2bound1}) yields the statement of the proposition.
\end{proof}
%\phantom{.}\hfill$Q.E.D.$\par\noindent

\begin{proof}[Proof of Proposition~\ref{prop:atoalpha}]
The local boundedness of the map
\begin{equation}
  \left(c, (\sigma_n)_{n\in\NN}, (a_n)_{n\in\NN}\right) \;\mapsto\; (\alpha_n)_{n\in\NN}
\end{equation}
from $\CC\times\mathcal{S}_\CC\times(\ell^2_1\tensor\CC)$ to $\ell^2_1\tensor\CC$ is implied by the following: For each $C_1$, there exists $C_2$ such that, for all $c, (\sigma_n), (a_n)$ with $|c|<C_1$, $\|(\sigma_n)\|_{\ell^2}<C_1$ and $\|(a_n)\|_{\ell^2_1}<C_1$, one has $\|(\alpha_n)\|_{\ell^2_1}<C_2$.

Given the conditions $|c|<C_1$ and $\|(\sigma_n)\|_{\ell^2}<C_1$, Lemma~\ref{lemma:Kbounds}, together with Young's generalized inequality, implies that the linear map $(a_n)\mapsto(\alpha_n)$ is bounded from $\ell^2_1\tensor\CC$ to itself with bound $2C$.  The additional condition $\|(a_n)\|_{\ell^2_1}<C_1$ then implies $\|(\alpha_n)\|_{\ell^2_1}<2CC_1$.
\end{proof}

  According to the discussion after the statement of the theorem, Proposition~\ref{prop:atoalpha} completes the proof of Theorem~\ref{thm:analytic}.

\bigskip
\noindent{\bfseries Acknowledgement.}
This material is based upon work supported by the National Science Foundation under Grant No. DMS-1814902 and by UK EPSRC under grant EP/P005985/1.

\end{document}